\documentclass[11pt,reqno]{amsart}
\usepackage[T1]{fontenc}
\usepackage[utf8]{inputenc}
\usepackage{lmodern,amsmath,amssymb,mathtools,microtype}
\calclayout
\usepackage[colorlinks, linkcolor=blue,citecolor=blue]{hyperref}
\hypersetup{pdftitle={Metric Poissonian pair correlation for n + log n},pdfauthor={}}
\numberwithin{equation}{section}
\newtheorem{theorem}{Theorem}[section]
\newtheorem{proposition}[theorem]{Proposition}
\newtheorem{lemma}[theorem]{Lemma}

\newtheorem{remark}[theorem]{Remark}
\newcommand{\ee}{\mathrm e}
\newcommand{\dd}{\mathrm d}
\newcommand{\RR}{\mathbb R}
\newcommand{\ZZ}{\mathbb Z}
\newcommand{\NN}{\mathbb N}

\newcommand{\loc}{\mathrm{loc}}

\allowdisplaybreaks[2]
\title[Metric pair correlation for $n+\log n$]{%
Metric Poissonian Pair Correlation Property\\
Part 1: The Case \(n+\log n\)}
\author[B. Tan, Z. Zhang, Q.L. Zhou]{Bo Tan, \ \  Zhen Zhang, \ \ Qing-Long Zhou}
\date{}
\begin{document}
\raggedbottom
\begin{abstract}
We prove that $(\alpha(n+\log n))_{n\ge1}$ has Poissonian pair correlation for Lebesgue almost every real $\alpha$. This solves a Problem  posed by Aistleitner, El-Baz and Munsch [Open Problem 2, Geom. Funct. Anal. 2021]. 
\end{abstract}
\maketitle

\section{Introduction}

This is the first of two papers devoted to metric results for the Poissonian pair correlation (PPC) property. Together, they address two open problems concerning the metric behavior of PPC: the \(n+\log n\) and \(n\log n\) cases. The present paper treats the \(n+\log n\) case, while the companion paper handles the \(n\log n\) case. We first review the background and significance of both problems to place our result in a broader context; the companion paper will refer back to this one, thereby connecting the two results and completing the overall program.

A real sequence \((a_n)_{n\in\mathbb{N}}\) is uniformly distributed modulo one if for any interval \([a, b] \subseteq [0, 1]\),
\[
\lim_{N\to\infty} \frac{\sharp\{1 \le n \le N \colon a_n \bmod 1 \in [a, b]\}}{N} = b - a,
\]
where \(\sharp\) denotes the cardinality of a finite set. Introduced by Weyl \cite{Weyl}, this concept connects deeply with number theory, fractal geometry, and probability theory. Classical results by Weyl, Fej\'er, and Csillag \cite{Weyl,C30,KN74} show that the sequence \((\alpha n^\theta)\) is uniformly distributed modulo 1 if \(\alpha \neq 0\) and \(\theta > 0\) is non-integral, or if \(\theta > 0\) and \(\alpha\) is irrational. Furthermore, Koksma's general metric theorem \cite{K35} (see also Theorem 4.3 in \cite{KN74}) yields that \((\alpha(n+\log n))\) and \((\alpha n\log n)\) are uniformly distributed modulo 1 for almost every \(\alpha\in\RR\).

The study of finer-scale statistics for such sequences was initiated by Rudnick and Sarnak \cite{RS}. The pair correlation function is defined by
\[
 R_N(s,(a_n))=\frac1{N}{\sharp\left\{1\le m\ne n\le N:
       \|a_m-a_n\|_{\RR/\ZZ}\le\frac{s}{N}\right\}},
\]
where  \(\|\cdot\|_{\RR/\ZZ}\) denotes the distance to the nearest integer. We say that \((a_n)\) has the PPC property if, for every \(s>0\),
\begin{equation*}
\lim_{N\to\infty}R_N(s,(a_n))=2s.
\end{equation*}
Originating in quantum physics, this notion plays a central role in the Berry--Tabor conjecture \cite{BT}; we refer to \cite{Mar01} for further background.

PPC implies uniform distribution, as proved independently by Aistleitner, Lachmann, and Pausinger~\cite{ALP}, and by Grepstad and Larcher~\cite{GL}. In contrast, many important classes of uniformly distributed sequences, such as van der Corput and Kronecker sequences, fail to have PPC (see \cite{LS20,PS19,Walker}).

\medskip

\noindent\textbf{Explicit examples.} Proving the PPC property for deterministic sequences is notoriously challenging; consequently, explicit results in this direction remain scarce. El-Baz, Marklof and Vinogradov \cite{EMV} provided the first explicit example by  proving that \((\sqrt n)\), after deleting the terms indexed by perfect squares, has PPC. Lutsko, Sourmelidis and Technau \cite{LST} showed that \(\alpha n^\theta\) has PPC for \(0 < \theta < 14/41\) and \(\alpha > 0\), a result later improved to \(0 < \theta < 43/117\) by Radziwi\l\l{} and Shubin \cite{RadSh}. Furthermore, Lutsko and Technau \cite{LTslow} demonstrated that the slowly growing sequence \((\alpha(\log n)^A)\) has PPC for all \(\alpha > 0\) and \(A > 1\). Recently, Hauke \cite{H25} proved that for any badly approximable \(\alpha\) and any sequence \((a_n) \) of rough numbers with explicit roughness parameters, the sequence \((a_n \alpha) \) has PPC.
\medskip

\noindent\textbf{Metric PPC.} A complementary direction of research focuses on the metric theory. We say that a sequence \((a_n)\) satisfies the metric PPC property if
\[
\lim_{N\to\infty}R_N(s,(a_n\alpha))=2s
\]
holds for almost every \(\alpha \in \mathbb{R}\) and every \(s > 0\). In this setting the problem is often more tractable for rapidly growing sequences \((a_n)\). Rudnick and Sarnak \cite{RS} established metric PPC for \((n^d)\) for  integers \(d \geq  2\); for the quadratic case, alternative proofs were later provided by Marklof and Str\"ombergsson \cite{MS03} and Heath-Brown \cite{HB}. Metric PPC is also known for all real-valued lacunary sequences \((a_n)\) \cite{RZlac,RT20}. A general sufficient condition based on a counting problem was obtained by Rudnick and Technau \cite{RT20}.

Over the past decade, significant progress has been made in establishing sufficient conditions for a sequence \((a_n) \) to be metric PPC. 
The truncated additive energy plays a central role in this study, which is defined as
\[
\sharp \mathcal{E}_{N}(0) = \sharp \{1 \leqslant n_1, n_2, n_3, n_4 \leqslant N : a_{n_1} + a_{n_2} = a_{n_3} + a_{n_4}\}.
\]
 Aistleitner, Larcher and Lewko \cite[Theorem~1]{ALL} proved that any sequence of distinct integers \((a_n)\) is metric Poissonian provided that there exists a fixed \(\delta>0\) such that
\begin{equation}\label{upp}
\sharp \mathcal{E}_{N}(0)\ll N^{3-\delta}.
\end{equation}
The exponent \(3-\delta\) is sharp: trivially \(\sharp \mathcal{E}_{N}(0) \leq N^3\), and Bourgain \cite[Appendix]{ALL} showed that \((a_n)\) fails to be metric Poissonian whenever
\begin{equation}\label{sup}
\limsup_{N\to\infty} \frac{\sharp \mathcal{E}_{N}(0)}{N^3} > 0. 
\end{equation}
A gap remains between conditions (\ref{upp}) and (\ref{sup}),  and it is an open problem whether the metric Poissonian property admits a purely additive or combinatorial characterization. Indeed, Aistleitner, Lachmann
 and Technau \cite{ALT} demonstrated that a criterion based solely on additive energy is insufficient. For progress on related extremal problems, we refer the reader to \cite{BM25,BCGW,BW,LTsets,LSenergy}.

Extending these results to real sequences is significantly more challenging. Aistleitner, El-Baz, and Munsch \cite{AEM} developed criteria based on counting approximate four-term relations. Their approach combines Fourier analysis, approximate additive energy, and estimates for moments of the Riemann zeta function. Specifically, they showed that the relevant lattice point counting problem can be controlled by the \(\gamma\)-approximate additive energy, denfined by 
\[
\sharp\mathcal{E}_{N}(\gamma) = \sharp\{1 \leq n_1, n_2, n_3, n_4 \leq N : |a_{n_1} - a_{n_2} + a_{n_3} - a_{n_4}| < \gamma\}.
\]
\begin{theorem}[\cite{AEM}, Theorem 2]\label{cri}
Let \((a_n) \) be a sequence of positive real numbers for which there exists \(c>0\) such that \(a_{n+1} - a_n \ge c\), \(n \ge 1\). Assume that there exists some \(\delta>0\) such that for all \(\eta>0\) we have
\begin{equation*}
\sharp\mathcal{E}_{N}(\gamma) \ll_{\eta, \delta} N^{2+\eta} + \gamma N^{3-\delta}
\end{equation*}
as \(N \to \infty\), uniformly for \(\gamma \in (0, 1]\). Then the sequence \((a_n) \) has metric PPC.
\end{theorem}
As an application of their criteria, Aistleitner, El-Baz, and Munsch \cite{AEM} proved that \((n^\theta) \) has metric PPC for every fixed \(\theta>1\). In addition, they posed three model problems, denoted as Open Problems 1--3, concerning sequences
\[
a_n = n^\theta \quad (0<\theta<1), \qquad a_n = n+\log n, \qquad a_n = n\log n.
\]
These instances lie outside the range covered by Theorem \ref{cri} (see \cite[Section 9]{AEM} for an explanation), and include sequences exhibiting sublinear or nearly linear growth.
The first of these was subsequently resolved by Rudnick and Technau \cite{RT}, who proved that \((n^\theta)\) satisfies metric PPC for \(0 < \theta < 1\). The remaining two models investigate how local pair statistics change under a logarithmic perturbation of a linear sequence or a logarithmic increase in its growth. The present paper and its forthcoming companion~\cite{PaperII} address the \(n+\log n\) and \(n\log n\) cases, respectively.

\begin{theorem}\label{thm:main}
The sequence \((n+\log n)\) has metric PPC. 
\end{theorem}

For a real-valued even function \(g\in C_c^\infty(\RR)\), set
\begin{equation}\label{eq:smooth-statistic}
 R_N(g,\alpha):=\frac1N
 \sum_{1\le m\ne n\le N}\sum_{\ell\in\ZZ}
 g\bigl(N(\alpha(a_n-a_m)+\ell)\bigr).
\end{equation}
The following moment estimates are crucial for the proof of Theorem~\ref{thm:main}.

\begin{theorem}\label{thm:mean-square}
Let \(I=[a,b]\) with \(0<a<b<\infty\), and let \(g\in C_c^\infty(\RR)\) be real-valued and even. Then, for \(N\ge2\),
\begin{equation*}
 \int_I\left|R_N(g,\alpha)-\int_\RR g(x)\dd x\right|^2\dd\alpha
 \ll_{I,g} N^{-1/3}\log^2(2N).
\end{equation*}
\end{theorem}

Notably, a general framework for establishing metric PPC for real sequences was recently provided by Kerr and Wang \cite{KW}. For completeness, we recall their result below.

\begin{theorem}[\cite{KW}, Theorem 2.1]\label{thm:KW}
Let \((a_n) \) be a sequence of positive real numbers for which there exist \(c>0\) and \(\eta \in (2/3, 1]\) such that
\begin{equation*}
    a_{n+1} - a_n \ge \frac{c}{n^{1-\eta}}.
\end{equation*}
Assume that there exist \(\delta > 0\) and \(\kappa \in [0, 1)\) such that
\begin{equation*}
    \sharp\mathcal{E}_{N}(1/N) \ll N^{3-\kappa-\delta}
\end{equation*}
and
\begin{equation}\label{condi}
    \sum_{\substack{1\le n_1<n_2\le N \\ a_{n_2}-a_{n_1}\ge 1}} \frac{1}{\sqrt{a_{n_2}-a_{n_1}}} \ll N^{1+\kappa/2}.
\end{equation}
Then \((a_n)\) has metric PPC.
\end{theorem}

\begin{remark}
The criterion in Theorem \ref{thm:KW} fails for our sequence \(a_n = n+\log n\). Indeed, a direct computation yields that the sum in condition (\ref{condi}) is exactly of order \(\asymp N^{3/2}\). To satisfy the inequality, we would require \(\kappa \ge 1\), which contradicts the requirement \(\kappa \in [0,1)\).  Therefore, the general criterion of Kerr and Wang cannot be directly applied here.
\end{remark}

\noindent{\textbf{Strategy.}}
We divide the positive frequencies and positive differences into dyadic ranges and apply Mellin inversion before estamating the second moment. This approach separates the frequency sum from the sum over differences while retaining the localization in \(\alpha\), allowing us to control the frequency sum via a third derivative estimate. Subsequently, Parseval's identity in \(\alpha\), followed by a mean-square estimate in logarithmic coordinates, reduces the problem to a four-point count at scale \(N^{-1}\).

More precisely, Proposition~\ref{prop:reduction} gives
\[
 \int_I\left|R_N(g,\alpha)-\int_\RR g(\alpha)\dd \alpha\right|^2\dd\alpha
 \ll_{I,g}N^{-\frac73}\sharp\mathcal{E}_N(2/N)+N^{-2}\log^4(2N).
\]
A central contribution to this second-moment estimate is the establishment of the following \(\gamma\)-approximate additive energy bound:
\begin{equation}\label{eq:intro-energy}
 \sharp\mathcal{E}_{N}(\gamma)\ll N^2+\gamma N^3\log^2(2N),
 \qquad N\ge2,\quad 0<\gamma\le1.
\end{equation}
 Crucially, at the fine scale \(\gamma=2/N\), the bound yields
\[
 \sharp\mathcal{E}_N(2/N)\ll N^2\log^2(2N),
\]
which provides the required power saving in Theorem~\ref{thm:mean-square}. 

The authors would like to express their gratitude to Aistleitner, El-Baz, and Munsch for their insightful remarks regarding this problem (see Open Problem 2 in \cite{AEM}), which inspired and guided us to establish the desired energy estimate \eqref{eq:intro-energy}.

\medskip

\noindent\textbf{Organization.} Section~\ref{sec:energy} is devoted to proving the four-point estimate. Section~\ref{sec:reduction} establishes the analytic reduction, including the third derivative estimate in the required form, while Section~\ref{sec:completion} provides the proofs of Theorems~\ref{thm:mean-square} and~\ref{thm:main}.

\medskip

\noindent\textbf{Notation.} We denote \(\NN=\{1,2,\ldots\}\) and \(\ee(t)=\exp(2\pi it)\). By dyadic numbers we mean the powers of two, i.e.,  the integers in
\(\{1,2,4,8,\ldots\}\). The relation \(U\ll V\) means \(|U|\le CV\) for some absolute constant \(C\); subscripts indicate permitted dependencies. Throughout the analytic estimates, \(I\) and \(g\) are fixed.

\section{A four-point counting estimate}\label{sec:energy}

From now on, we always write  \(a_n = n + \log n\). For \(h \in \mathbb{N}\) and real \(n \ge 1\), the gap function is given by
\begin{equation*}
    \Delta_h(n) := a_{n+h} - a_n = h + \log\left(1 + \frac{h}{n}\right).
\end{equation*}
Since \(\log(1+x) \le x\) for all \(x > -1\), it is readily checked that \(   h < \Delta_h(n) \le 2h.\)

For \(J, H \in \mathbb{N}\) and \(\gamma \in (0,1]\), we   set
\begin{align}\label{c4}
\mathcal{E}_{J,H}(\gamma)
&:= \left\{ (n,h_1,m,h_2) \in \mathbb{N}^4 : 
\begin{array}{l}
J \le n,m < 2J, H \le h_1,h_2 < 2H,\\[3pt]
\quad\qquad |\Delta_{h_1}(n) - \Delta_{h_2}(m)| \le \gamma
\end{array} \right\}.
\end{align}
The following lemma provides an estimate for the cardinality of 
\(\mathcal{E}_{J,H}(\gamma)\).
\begin{lemma}\label{local-energy} We have
\[\sharp\mathcal E_{J,H}(\gamma)
 \le 5JH+40\gamma J^2(J+H).\]
\end{lemma}

\begin{proof}
Fix $(n,h_1)\in\mathbb{N}^2$ with $n\in[J,2J)$ and $h_1\in[H,2H)$. 
For any $m\in[J,2J)$ and $h_2\in[H,2H)$, 
a direct calculation gives
\[
\left|\log\left(1+\frac{h_1}{n}\right)
- \log\left(1+\frac{h_2}{m}\right)\right| < \log 4.
\]
Since 
\[
\Delta_{h_1}(n)-\Delta_{h_2}(m)
= h_1-h_2
+ \log\left(1+\frac{h_1}{n}\right)
- \log\left(1+\frac{h_2}{m}\right),
\]  it follows that if $|\Delta_{h_1}(n)-\Delta_{h_2}(m)|\le\gamma$, then
$|h_1-h_2|\le \gamma+\log 4<3$ (since $\gamma\le 1$). 
Since \(h_1,h_2\) are integers, the preceding condition leaves at most five possible values of \(h_2\) for each fixed \(h_1\).

Now fix such an $h_2$. Consider the function
\[
m\longmapsto \Delta_{h_2}(m)
= h_2+\log\left(1+\frac{h_2}{m}\right),
\qquad m\in[J,2J).
\]
It is strictly decreasing. 
Since $m\in[J,2J)$ and $h_2\in[H,2H)$, we have
\[
|\Delta_{h_2}'(m)|
= \frac{h_2}{m(m+h_2)}
\ge \frac{H}{4J(J+H)}.
\]
Suppose $m_1,m_2\in[J,2J)$ both satisfy
$|\Delta_{h_1}(n)-\Delta_{h_2}(m_i)|\le\gamma$ for $i=1,2$. 
Then $|\Delta_{h_2}(m_1)-\Delta_{h_2}(m_2)|\le 2\gamma$. 
By the mean value theorem, there exists $\xi$ between $m_1$ and 
$m_2$ such that
\[
|\Delta_{h_2}(m_1)-\Delta_{h_2}(m_2)|
= |\Delta_{h_2}'(\xi)|\,|m_1-m_2|
\ge \frac{H}{4J(J+H)}\,|m_1-m_2|.
\]
Hence
\[
|m_1-m_2|
\le \frac{2\gamma\cdot 4J(J+H)}{H}
= \frac{8\gamma J(J+H)}{H}.
\]
Therefore all admissible real values of $m$ lie in an interval of 
length at most $8\gamma J(J+H)/H$, which contains at most
\[
1+\frac{8\gamma J(J+H)}{H}
\]
integers. Consequently, for fixed $(n,h_1)$, the number of pairs 
$(h_2,m)$ is at most
\[
5\left(1+\frac{8\gamma J(J+H)}{H}\right).
\]
Finally, there are $JH$ choices for $(n,h_1)$. Multiplying gives
\[
\sharp\mathcal{E}_{J,H}(\gamma)
\le JH \cdot 5\left(1+\frac{8\gamma J(J+H)}{H}\right)
= 5JH+40\gamma J^2(J+H).
\]
\end{proof}

For $N\ge2$ and $0<\gamma\le1$, recall that
\begin{equation}\label{cn}
  \mathcal{E}_{N}(\gamma)=
  \left\{(n_1,n_2,n_3,n_4)\in\{1,\ldots,N\}^4:
  |a_{n_1}-a_{n_2}+a_{n_3}-a_{n_4}|<\gamma\right\}.
\end{equation}

\begin{proposition}\label{energy}
For $N\ge2$ and $0<\gamma\le1$,
\begin{equation*}
 \sharp\mathcal{E}_N(\gamma)
 \ll N^2+\gamma N^3\log^2(2N),
\end{equation*}
with an absolute implied constant independent of $N$ and $\gamma$.
In particular,
\begin{equation*}
 \sharp\mathcal{E}_N(2/N)
 \ll N^2\log^2(2N).
\end{equation*}
\end{proposition}

\begin{proof}
The latter assertion follows directly from the former, which we now prove.
We examine the signs of the differences \(a_{n_1} - a_{n_2}\) and \(a_{n_3} - a_{n_4}\) for a quadruple \((n_1, n_2, n_3, n_4) \in \mathcal{E}_N(\gamma)\). 

Given that \(\gamma \le 1\), if these two differences have the same sign, or if one is zero while the other is non-zero, the absolute value of their sum is at least \(1\). Hence, it cannot be strictly less than \(\gamma\). Therefore, such quadruples cannot belong to the set \(\mathcal{E}_N(\gamma)\). 
Consequently, the valid quadruples fall into exactly two cases:
\begin{itemize}
    \item Both differences are zero, i.e., \(a_{n_1} - a_{n_2} = 0\) and \(a_{n_3} - a_{n_4} = 0\). This trivial case contributes \(N^2\) to the cardinality under consideration.
    \item The two differences have opposite signs. By symmetry, we may assume without loss of generality that \(a_{n_1} - a_{n_2} > 0\) and \(a_{n_3} - a_{n_4} < 0\).
\end{itemize}

We introduce new variables \(n, h_1, m, h_2 \in \mathbb{N}\) and recall the difference operator \(\Delta_h(n) := a_{n+h} - a_n\). We then define the set
\[
 \mathcal{E}_{N}^{\ast}(\gamma) = \left\{(n,h_1,m,h_2)\in\mathbb{N}^4 : 
 \begin{array}{l}
   n+h_1\le N,\ m+h_2\le N, \\
  \ \  |\Delta_{h_1}(n)-\Delta_{h_2}(m)|<\gamma
 \end{array}
 \right\}.
\]

We   bound $\sharp\mathcal{E}_{N}^{\ast}(\gamma).$
 For each \(k\in\mathbb{N}\), define the \(k\)-th bin:
\[
 C_k=\left\{(n,h)\in\mathbb{N}^2: n+h\le N,\quad k\gamma\le \Delta_{h}(n)<(k+1)\gamma\right\}.
\]
Since \(|\Delta_{h_1}(n)-\Delta_{h_2}(m)| < \gamma\) and each bin has width \(\gamma\), \(\Delta_{h_1}(n)\) and \(\Delta_{h_2}(m)\) must fall into the same bin or adjacent bins. Thus, the number of such pairs is bounded by:
\begin{equation}\label{eq:adjacent-bins}
 \sharp\mathcal{E}_{N}^{\ast}(\gamma)
 \le\underbrace{\sum_{k=1}^{\infty} (\sharp C_k)^2}_{\text{same bin}}+\underbrace{2\sum_{k=1}^{\infty} (\sharp C_k \sharp C_{k+1})}_{\text{adjacent bins}}\le 3\sum_{k=1}^{\infty} (\sharp C_k)^2.
\end{equation}

For dyadic \(J,H\le N\), the set \(C_{J,H,k}\) is defined as:
\begin{equation*}
 C_{J,H,k} = \left\{ (n,h)\in\mathbb{N}^2 : 
 \begin{array}{l}
   J \le n < 2J, \  H \le h < 2H, \\
   n+h \le N,  
   k\gamma \le \Delta_{h}(n) < (k+1)\gamma
 \end{array}
 \right\}.
\end{equation*} Then,
\[
 \sharp C_k= \sum_{J,H}\sharp C_{J,H,k}\quad\text{and} \quad
 \sum_{k=1}^{\infty} (\sharp C_{J,H,k})^2\le\sharp\mathcal{E}_{J,H}(\gamma),
\]
with  $\mathcal{E}_{J,H}(\gamma)$ defined in   \eqref{c4}. Here and throughout, all sums over \(H\) and \(J\) run over dyadic integers in the appropriate ranges. 
By Minkowski's inequality (the triangle inequality in \(\ell^2\)), we have
\begin{equation}\label{Min}
 \left(\sum_k  (\sharp C_k)^2\right)^{1/2}
 \le\sum_{J,H}\left(\sum_k (\sharp C_{J,H,k})^2\right)^{1/2}
 \le\sum_{J,H}\sqrt{\sharp\mathcal{E}_{J,H}(\gamma)}.
\end{equation}
Applying Lemma~\ref{local-energy} and the inequality \(\sqrt{J+H}\le\sqrt J+\sqrt H\), we obtain
\begin{align*}
 \sum_{J,H}\sqrt{\sharp\mathcal{E}_{J,H}(\gamma)}
 &\ll\sum_{J,H}\sqrt{JH}
       +\sqrt\gamma\sum_{J,H}J\sqrt{J+H}\\
 &\ll N+\sqrt\gamma\left(
       \sum_{J\le N}J^{3/2}\sum_{H\le N}1
       +\sum_{J\le N}J\sum_{H\le N}\sqrt H\right)\\
 &\ll N+\sqrt\gamma\,N^{3/2}\log(2N).
\end{align*}
  Squaring both sides and using  \eqref{eq:adjacent-bins}-\eqref{Min}, we obtain that 
\(\sharp\mathcal{E}_{N}^{\ast}(\gamma)\ll N^2+\gamma N^3\log^2(2N).\) 

\medskip

Combining both cases, we conclude that 
\begin{equation*}
  \sharp\mathcal{E}_{N}(\gamma) = N^2 + 2\sharp\mathcal{E}_{N}^{\ast}(\gamma)\ll N^2+\gamma N^3\log^2(2N).  
\end{equation*}
\vskip -0.6cm
\end{proof}

\section{Localized Mellin reduction}\label{sec:reduction}

Fix \(I=[a,b]\subset(0,\infty)\) and let \(g\in C_c^\infty(\mathbb{R})\) be an even, real-valued function. The Fourier transform of \(g\) is defined as
\[
 \widehat g(\xi)=\int_\mathbb{R} g(\alpha)e(-\xi \alpha) \dd\alpha, \qquad \xi\in\mathbb{R},
\]
and the \(L^2(I)\) norm is given by \(\|g\|_{L^2(I)}=\left(\int_I|g(\alpha)|^2 \dd\alpha\right)^{1/2}\).

The aim of this section is to establish the following 
\(L^2\)-estimate, which reduces the \(L^2(I)\) integral to a combinatorial counting problem involving \(\mathcal{E}_N(2/N)\).
\begin{proposition}\label{prop:reduction}
For every integer \(N\ge2\),
\begin{equation*}
 \int_I|R_N(g,\alpha)-\widehat g(0)|^2\dd\alpha
 \ll_{I,g}N^{-7/3}\sharp\mathcal{E}_N(2/N)+N^{-2}\log^4(2N).
\end{equation*}
\end{proposition}

\subsection{Fourier expansion and dyadic decomposition}

By the definition of \(R_N(g,\alpha)\) in \eqref{eq:smooth-statistic} and the evenness of \(g\), we obtain
\[
 R_N(g,\alpha)=\frac2N
 \sum_{\substack{n,h\ge1\\n+h\le N}}\sum_{\ell\in\mathbb{Z}}
       g\bigl(N(\alpha \Delta_h(n)+\ell)\bigr).
\]
Poisson summation formula yields
\[
 R_N(g,\alpha)=\frac2{N^2}\sum_{j\in\mathbb{Z}}\widehat g(j/N)
        \sum_{\substack{n,h\ge1\\n+h\le N}}e(\alpha j\Delta_h(n)).
\]
The zero frequency (\(j=0\)) contributes \((1-N^{-1})\widehat g(0)\). Since \(\widehat g\) is real and even, the terms at frequencies \(\pm j\) occur in complex conjugate pairs. Moreover, the rapid decay of \(\widehat g\) ensures that for any \(A>1\),
\[
 \sum_{j>N^2}|\widehat g(j/N)|\ll_{g,A}N^{2-A}.
\]
Hence, uniformly for \(\alpha\in\mathbb{R}\), we obtain
\begin{equation*}
 R_N(g,\alpha)-\widehat g(0)
 =\operatorname{Re}R_N^{\ast}(g,\alpha)+O_g(N^{-1}).
\end{equation*}
where
\begin{equation}\label{eq:block-definition}
 R_N^{\ast}(g,\alpha)=\frac4{N^2}\sum_{1\le j\le N^2}\widehat g(j/N)
       \sum_{\substack{n,h\ge1\\n+h\le N}}e(\alpha j\Delta_h(n)).
\end{equation}

For dyadic \(H\le N\), we define the block
\[
 \mathcal B_H=\{(n,h)\in\mathbb{N}^2:H\le h<2H,\ n+h\le N\},
\]
which satisfies \(\sharp\mathcal B_H\le NH\) and \(H\le \Delta_h(n)\le4H\). For dyadic \(J\le N^2\), we further decompose the frequencies by introducing 
\begin{equation*}
R_{H,J}^{\ast\ast}(g;\alpha)=
 \sum_{\substack{J\le j<2J\\j\le N^2}}\widehat g(j/N)
       \sum_{(n,h)\in\mathcal B_H}e(\alpha j\Delta_h(n)).
\end{equation*}
Thus we have 
\begin{equation*}
R_{N}^{\ast}(g;\alpha)=\frac4{N^2}\sum_{H,J}R_{H,J}^{\ast\ast}(g;\alpha).
\end{equation*}

Hence,
Minkowski's inequality yields
\begin{equation}\label{eq:bound-RN}
\begin{aligned}
 \|R_N(g,\alpha)-\widehat g(0)\|_{L^2(I)}&\le\|R_{N}^{\ast}(g,\alpha)\|_{L^2(I)}+O_g(N^{-1})\\ 
 &\le\frac4{N^2}\sum_{H,J}\|R_{H,J}^{\ast\ast}(g;\alpha)\|_{L^2(I)}+O_{g,I}(N^{-1}).  
\end{aligned}
\end{equation}

\subsection{Mellin localization and Parseval's identity}

This subsection is devoted to estimating the \(L^2(I)\) norm of \(R_N^{\ast}(g,\alpha)\) from above.  For \(t\in\mathbb{R}\), we define
\begin{equation*}
 D_J(t)=\sum_{\substack{J\le j<2J\\j\le N^2}}
                  \widehat g(j/N)\exp(it\log j)\quad \text{and}\quad
 Z_H(t)=\sum_{(n,h)\in\mathcal B_H}\exp(it\log \Delta_h(n)).
\end{equation*}

\begin{lemma}\label{lem:mellin}
For all $H,J$ with $H\le N$ and $J\le N^2$, we have  
\begin{equation*}
 \|R_{H,J}^{\ast\ast}(g;\alpha)\|_{L^2(I)}
 \ll_{I,g}(JH)^{-1/2}
       \left(\int_{\pi aJH/2}^{64\pi bJH}|D_J(t)Z_H(t)|^2\dd t\right)^{1/2}+N.
\end{equation*}
\end{lemma}

\begin{proof}
Put \(X=JH\) and choose a smooth function \(\psi\in C_c^\infty((1/2,16))\) such that \(0\le\psi\le1\) and \(\psi=1\) on \([1,8]\). We readily check that \(|\psi(x/X)|^2/x\le2/X\) for all \(x>0\).

For \(\alpha,t\in\mathbb{R}\), the Mellin-type integral obtained by evaluating the Mellin transform of
\(x\mapsto\psi(x/X)\ee(\alpha x)\) at \(s=-it\) is
\begin{equation*}
\begin{aligned}
K_X(\alpha,t)&=\int_0^\infty \psi(x/X)\ee(\alpha x)x^{-it}\frac{\dd x}{x} \\
&=\exp(-it\log X)\int_\mathbb{R}\psi(e^u) \exp\bigl(i(2\pi\alpha Xe^u-tu)\bigr)\dd u,
\end{aligned}
\end{equation*}
where the second equality follows from the substitution \(x=Xe^u\).
We record two properties of  \(K_X(\alpha,t)\):
\begin{enumerate}
   
    \item[(1)]  Mellin inversion (equivalently, applying Fourier inversion and then substituting
back \(e^u=x/X\)) yields
    \[
     \psi(x/X)\ee(\alpha x) = \frac1{2\pi}\int_\RR K_X(\alpha,t)\exp(it\log x)\dd t.
    \]
    
    \item[(2)]  On the summation range in \eqref{eq:block-definition}, we have \(X\le j\Delta_h(n)\le8X\), and thus  \(\psi(j\Delta_h(n)/X)=1\). Hence,
\[
 R_{H,J}^{\ast\ast}(g;\alpha) = \sum_{\substack{J\le j<2J\\j\le N^2}} \widehat g(j/N) \sum_{(n,h)\in\mathcal B_H} \psi(j\Delta_h(n)/X)\ee(\alpha j\Delta_h(n)).
\]
Using (1)  with \(x=j\Delta_h(n)\) 
we obtain
\begin{align*}
R_{H,J}^{\ast\ast}(g;\alpha)
&= \frac1{2\pi}\int_\RR K_X(\alpha,t) \left(\sum_{\substack{J\le j<2J\\j\le N^2}} \widehat g(j/N)\exp(it\log j)\right) \\
&\qquad \times \left(\sum_{(n,h)\in\mathcal B_H}\exp(it\log \Delta_h(n))\right)\dd t \\
&= \frac1{2\pi}\int_\RR K_X(\alpha,t)D_J(t)Z_H(t)\dd t.
\end{align*}

\end{enumerate}

Setting \(\mathcal{T}_X=[\pi aX/2,64\pi bX]\), we split the integration range of
the last integral above into the non-stationary region \(t\notin\mathcal{T}_X\) and
the localized region \(t\in\mathcal{T}_X\).

\medskip

\noindent\textbf{Part 1: The non-stationary region.} 
For \(\alpha\in I\) and \(u\) in the support of \(\psi(e^u)\), we have
\(
 \pi aX\le2\pi\alpha Xe^u\le32\pi bX.
\)
Consequently, if \(t\notin\mathcal{T}_X\), the phase \(\phi(u)=2\pi\alpha Xe^u-tu\) satisfies
\[
 |\phi'(u)|\gg_I X+|t|,\qquad
 |\phi^{(r)}(u)|\ll_{I,r}X\quad(r\ge2).
\]
Integrating by parts twice, using the compact support of the amplitude, we obtain
\begin{equation*}
 |K_X(\alpha,t)|\ll_I(X+|t|)^{-2}
 \qquad(\alpha\in I,\ t\notin\mathcal{T}_X).
\end{equation*}
Combining this with the trivial bounds \(|D_J(t)|\ll_g J\) and \(|Z_H(t)|\le NH\), we conclude 
\[
 \int_{\RR\setminus\mathcal T_X}|K_XD_JZ_H|\dd t
 \ll_{I,g}JNH\int_\RR(X+|t|)^{-2}\dd t
 \ll_{I,g}N.
\]
\smallskip

\noindent\textbf{Part 2: The localized main term.} 
We study 
\begin{align*}
R_{H,J}^{\ast\ast,\loc}(g;\alpha):=&\frac1{2\pi}
       \int_{\mathcal T_X}K_X(\alpha,t)D_J(t)Z_H(t)\dd t\\
       =&\int_0^\infty
                  \frac{\psi(x/X)}{x}A_X(x)\ee(\alpha x)\dd x,
    \end{align*}
where 
\[
 A_X(x)=\frac1{2\pi}\int_{\mathcal T_X}
                    D_J(t)Z_H(t)\exp(-it\log x)\dd t.
\] 

Applying Minkowski's inequality to the approximation
\[
R_{H,J}^{\ast\ast}(g;\alpha)=R_{H,J}^{\ast\ast,\loc}(g;\alpha)+O_{I,g}(N),
\]
we obtain
\[
 \|R_{H,J}^{\ast\ast}(g;\alpha)\|_{L^2(I)} \le \|R_{H,J}^{\ast\ast,\loc}(g;\alpha)\|_{L^2(I)} + O_{I,g}(N).
\]
It remains to estimate the \(L^2\) norm of \(R_{H,J}^{\ast\ast,\loc}(g;\alpha)\). Since it is the inverse Fourier transform of \(A_X(x)\psi(x/X)/x\), Parseval's identity yields
\begin{align}
 \int_\RR|R_{H,J}^{\ast\ast,\loc}(g;\alpha)|^2\dd\alpha
 &=\int_0^\infty\frac{|\psi(x/X)|^2}{x^2}|A_X(x)|^2\dd x \notag \\
 &\le \frac{2}{X} \int_0^\infty \frac{|A_X(x)|^2}{x}\dd x, \label{eq:first-parseval}
\end{align}
where we used the bound \(|\psi(x/X)|^2/x\le2/X\).

The function \(F(t) := \mathbf{1}_{\mathcal{T}_X}(t)D_J(t)Z_H(t)\) belongs to \(L^2(\mathbb{R})\). Under the substitution \(u = \log x/2\pi\), the definition of \(A_X\) becomes
\[
 A_X(e^{2\pi u}) = \frac{1}{2\pi}\int_{\mathcal{T}_X} D_J(t)Z_H(t)e(-tu)\dd t = \frac{1}{2\pi}\widehat{F}(u).
\]
Applying Parseval's identity again, 
we obtain
\[
 \int_\mathbb{R} |2\pi A_X(e^{2\pi u})|^2 \dd u = \int_{\mathcal{T}_X} |D_J(t)Z_H(t)|^2 \dd t.
\]
Dividing both sides by \(4\pi^2\) and changing the integration variable back to \(x = e^{2\pi u}\) (so \(\dd u = \dd x/(2\pi x)\)) we have
\begin{equation}\label{eq:second-parseval}
 \int_0^\infty |A_X(x)|^2\frac{\dd x}{x}
 =\frac1{2\pi}\int_{\mathcal{T}_X}|D_J(t)Z_H(t)|^2\dd t.
\end{equation}
Substituting \eqref{eq:second-parseval} into \eqref{eq:first-parseval} yields
\[
 \int_\RR|R_{H,J}^{\ast\ast,\loc}(g;\alpha)|^2\dd\alpha
 \le\frac1{\pi X}\int_{\mathcal T_X}|D_J(t)Z_H(t)|^2\dd t.
\]
Restricting the left-hand side to the interval \(I\), we conclude that
\[
 \|R_{H,J}^{\ast\ast}(g;\alpha)\|_{L^2(I)}^2
 \ll_{I,g} \frac1{X}\int_{\mathcal T_X}|D_J(t)Z_H(t)|^2\dd t + N^2.
\]

\end{proof}

\subsection{The integer frequency sum}

We first establish a third derivative estimate for the exponential sum, which is uniform in both the location \(u\) and the length \(M\) of the summation interval.

\begin{lemma}\label{lem:third-derivative}
Let \(u\in\mathbb{Z}\), \(M\in\mathbb{N}\), \(\lambda>0\) and \(A_0\ge1\).
Suppose that \(f:[u+1,u+M]\to\mathbb{R}\) is a \(C^3\) function satisfying
\[
 \lambda\le |f'''(x)|\le A_0\lambda
\]
throughout its domain. Then
\begin{equation*}
 \left|\sum_{m=1}^M e(f(u+m))\right|
 \ll_{A_0}M\lambda^{1/6}+M^{1/2}\lambda^{-1/6}.
\end{equation*}
\end{lemma}

\begin{proof}
The trivial bound suffices if \(\lambda\ge1\) or \(\lambda\le M^{-3}\). Thus, we may assume \(M^{-3}<\lambda<1\) and put \(Q=\lfloor\lambda^{-1/3}\rfloor\), so that \(1\le Q\le M\) and \(Q\asymp\lambda^{-1/3}\).

By the Weyl--van der Corput inequality (see \cite[Lemma~2.5 and (2.3.4)]{GK}) we have
\begin{equation}
    \label{eq:expsum}
    \left|\sum_{m=1}^M e(f(u+m))\right|^2
 \ll\frac{M^2}{Q}
 +\frac{M}{Q}\sum_{r=1}^{Q-1}
 \left|\sum_{m=1}^{M-r}
 e\bigl(f(u+m+r)-f(u+m)\bigr)\right|.
\end{equation}
Since \(f'''\) is continuous and non-vanishing, it is of constant sign. For the shifted phase \(f_r(x)=f(x+r)-f(x)\) on \([u+1,u+M-r]\), we have
\[
 f_r''(x)=\int_0^r f'''(x+v)\dd v
 \qquad \text{and} \qquad r\lambda\le |f_r''(x)|\le A_0r\lambda.
\]
By the second derivative estimate \cite[Theorem~2.2]{GK}, the inner sum in \eqref{eq:expsum} is bounded by \(O\bigl(A_0M(r\lambda)^{1/2}+(r\lambda)^{-1/2}\bigr)\). 
Therefore,
\begin{align*}
 \left|\sum_{m=1}^M e(f(u+m))\right|^2
 &\ll_{A_0}\frac{M^2}{Q}
       +M^2\lambda^{1/2}Q^{1/2}
       +M\lambda^{-1/2}Q^{-1/2}\\
 &\ll_{A_0}M^2\lambda^{1/3}+M\lambda^{-1/3}.
\end{align*}
Finally, taking square roots and applying the inequality \(\sqrt{a+b}\le\sqrt{a}+\sqrt{b}\) we obtain
\[
 \left|\sum_{m=1}^M e(f(u+m))\right|
 \ll_{A_0} M\lambda^{1/6}+M^{1/2}\lambda^{-1/6},
\]
which completes the proof.
\end{proof}

\begin{lemma}\label{lem:frequency}
For all $H,J$ with $H\le N$ and $J\le N^2$, 
\begin{equation*}
 \sup_{\pi aJH/2\le t\le64\pi bJH}|D_J(t)|
 \ll_{I,g}(1+J/N)^{-2}
       \bigl(J^{2/3}H^{1/6}+J^{5/6}H^{-1/6}\bigr).
\end{equation*}
\end{lemma}

\begin{proof}
For \(t>0\), define the phase function \(f(x) = t\log x/(2\pi)\). By partial summation, we may write \(D_J(t)\) as
\begin{equation*}
D_J(t) = \widehat g(J^{\ast}/N) \sum_{J \le j \le J^{\ast}} e(f(j)) - \int_J^{J^{\ast}} \left( \sum_{J \le j \le v} e(f(j)) \right) \frac{1}{N} \widehat g'(v/N) \dd v,
\end{equation*}
where \(J^{\ast} = \min(2J-1, N^2)\).
Taking absolute values yields
\begin{equation}\label{eq:DJ-bound-abel}
|D_J(t)| \le \sup_{J \le v < 2J} \left| \sum_{J \le j \le v} e(f(j)) \right| \left( \sup_{J \le x < 2J} |\widehat g(x/N)| + \frac{1}{N} \int_J^{2J} |\widehat g'(x/N)| \dd x \right).
\end{equation}

A direct calculation gives
\[
 \frac{t}{8\pi J^3} \le f'''(x) = \frac{t}{\pi x^3} \le \frac{t}{\pi J^3} \qquad (J \le x \le 2J). 
\]
Applying Lemma \ref{lem:third-derivative} with \(A_0=8\) and \(\lambda=t/(8\pi J^3)\) to the partial sums over any initial segment \([J,v]\) with \(J \le v < 2J\), we obtain
\begin{equation}\label{eq:exp-sum-bound}
 \sup_{J\le v<2J}\left|\sum_{J\le j\le v}e(f(j))\right|
 \ll J(t/J^3)^{1/6}+J^{1/2}(t/J^3)^{-1/6} = J^{1/2}t^{1/6}+Jt^{-1/6}.
\end{equation}
On the other hand, the rapid decay of \(\widehat g\) and \(\widehat g'\) implies
\begin{equation}\label{eq:g-bound}
 \sup_{J\le x\le2J}|\widehat g(x/N)| +\frac1N\int_J^{2J}|\widehat g'(x/N)|\dd x \ll_g(1+J/N)^{-2}.
\end{equation}
Substituting \eqref{eq:exp-sum-bound} and \eqref{eq:g-bound} into \eqref{eq:DJ-bound-abel}, we deduce
\[
 |D_J(t)| \ll_g (1+J/N)^{-2} \bigl(J^{1/2}t^{1/6}+Jt^{-1/6}\bigr).
\]

Finally, when \(t \asymp_I JH\), we have 
\[
 J^{1/2}t^{1/6} + Jt^{-1/6} \asymp_I J^{2/3}H^{1/6} + J^{5/6}H^{-1/6},
\]
which yields the proof.
\end{proof}

\subsection{The mean square of the gap sum}

\begin{lemma}\label{lem:gap-mean-square}
For all $H,J$ in the above ranges,
\begin{equation*}
 \int_{\pi aJH/2}^{64\pi bJH}|Z_H(t)|^2\dd t
 \ll_I H(J+N)\sharp\mathcal{E}_N(2/N).
\end{equation*}
\end{lemma}

\begin{proof}
Set \(T=64\pi bJH\). Since \(\exp(-t^2/T^2)\ge e^{-1}\) on the interval \([\pi aJH/2, 64\pi bJH]\), extending the integration domain and applying the Gaussian Fourier transform yield
\begin{align*}
 \int_{\pi aJH/2}^{64\pi bJH}|Z_H(t)|^2\dd t
 &\ll\int_\mathbb{R} e^{-t^2/T^2}|Z_H(t)|^2\dd t\\
 &=\sqrt\pi\,T\sum_{(n,h),(m,\eta)\in\mathcal B_H}
       \exp\left(-\frac{T^2}{4}
          \left(\log\frac{\Delta_h(n)}{\Delta_\eta (m)}\right)^2\right).
\end{align*}
Since both \(\Delta_h(n)\) and \(\Delta_\eta(m)\) lie in \([H,4H]\), applying the Mean Value Theorem to \(\log x\) on this interval yields
\[
 \left|\log\frac{\Delta_h(n)}{\Delta_\eta (m)}\right|
 =\left|\log \Delta_h(n) - \log \Delta_\eta (m)\right|
 \ge \frac{|\Delta_h(n)-\Delta_\eta (m)|}{4H}.
\]
It follows that, for a suitable constant \(\kappa:=\kappa(I)>0\) that depends only on \(I\),
\begin{equation}\label{eq:gaussian-energy}
 \int_{\pi aJH/2}^{64\pi bJH}|Z_H(t)|^2\dd t
 \ll_I JH\sum_{(n,h),(m,\eta)\in\mathcal B_H}
        \exp\bigl(-\kappa J^2|\Delta_h(n)-\Delta_\eta (m)|^2\bigr).
\end{equation}

As before, we partition the differences into bins of length \(1/N\) by setting
\[
 C_k=\{(n,h)\in\mathcal B_H: k/N\le \Delta_h(n)<(k+1)/N\},
 \qquad k\in\mathbb Z.
\]
If two elements lie in bins whose indices differ by \(r\), then their difference is at least \((|r|-1)_+/N\), where \(x_+=\max(x,0)\). Thus, the double sum on the right-hand side of \eqref{eq:gaussian-energy} is bounded by
\[
 \left(\sum_{k\in\mathbb{Z}} \sharp C_k \sharp C_{k+r}\right)\sum_{r\in\mathbb{Z}}
    \exp\left(-\kappa (J/N)^2((|r|-1)_+)^2\right).
\]
Since any two elements in the same bin differ by less than \(1/N\), we have
\begin{equation*}
 \sum_{k\in\mathbb{Z}} \sharp C_k^2\le \sharp\mathcal{E}_N(2/N),
\end{equation*}
where \(\mathcal{E}_N(2/N)\) is defined in \eqref{cn}.
Furthermore,  the Cauchy-Schwarz inequality implies
\[
 \sum_{k\in\mathbb{Z}} \sharp C_k \sharp C_{k+r} \le \left(\sum_{k\in\mathbb{Z}} \sharp C_k^2\right)^{1/2}\left(\sum_{k\in\mathbb{Z}} \sharp C_{k+r}^2\right)^{1/2} \le \sharp\mathcal{E}_N(2/N).
\]

Combining these bounds with \eqref{eq:gaussian-energy}, we obtain
 \begin{align*}
     \text{RHS of \eqref{eq:gaussian-energy}} &\ll_I JH \cdot \sharp\mathcal{E}_N(2/N) \cdot \sum_{r\in\mathbb{Z}}
    \exp\left(-\kappa (J/N)^2((|r|-1)_+)^2\right)\\
    &\ll_{I} JH \cdot \sharp\mathcal{E}_N(2/N) \cdot \left(1+\frac{N}{J}\right) = H(J+N) \sharp\mathcal{E}_N(2/N), 
 \end{align*}
where the last inequality follows from the estimate: for any \(A>0\),
\[
 \sum_{r\in\mathbb{Z}}e^{-A^2((|r|-1)_+)^2}
 \le3+2\int_0^\infty e^{-A^2x^2}\dd x
 =3+\frac{\sqrt\pi}{A}.
\]
\vskip -.82cm
\end{proof}

\bigskip

\subsection{Proof of Proposition~\ref{prop:reduction}}

\begin{proof}
By Lemmas \ref{lem:mellin}, \ref{lem:frequency}--\ref{lem:gap-mean-square},  we have
\begin{equation*}
 \|R_{H,J}^{\ast\ast}(g;\alpha)\|_{L^2(I)} \ll_{I,g} w_J\left(J^{2/3}H^{1/6}+J^{5/6}H^{-1/6}\right)\sqrt{\sharp\mathcal{E}_N(2/N)} + N,
\end{equation*}
where \(w_J = (1+J/N)^{-2}(1+N/J)^{1/2}\).

Summing over dyadic \(H\le N\) and \(J\le N^2\), we get
\[
 \sum_{H\le N} H^{1/6} \ll N^{1/6}, ~\sum_{H\le N} H^{-1/6} \ll 1, ~\text{and} ~ \sum_{J\le N^2} w_J J^r \ll_r N^r   \left(\frac12 < r < 2\right).
\]
Taking \(r=2/3\) and \(r=5/6\) in the above estimate, and noting that there are \(O(\log^2(2N))\) dyadic pairs \((H,J)\), we obtain
\[
 \sum_{H,J} \|R_{H,J}^{\ast\ast}(g;\alpha)\|_{L^2(I)} \ll_{I,g} N^{5/6}\sqrt{\sharp\mathcal{E}_N(2/N)} + N\log^2(2N). 
\]
By (\ref{eq:bound-RN}), we have 
\begin{align*}
 \|R_N(g,\alpha)-\widehat g(0)\|_{L^2(I)} 
 &\ll \frac{4}{N^2} \left( N^{5/6}\sqrt{\sharp\mathcal{E}_N(2/N)} + N\log^2(2N) \right) + O_g(N^{-1}) \\
 &\ll_{I,g} N^{-7/6}\sqrt{\sharp\mathcal{E}_N(2/N)} + N^{-1}\log^2(2N).
\end{align*}
Squaring both sides completes the proof.
\end{proof}

\section{Completion of the metric argument}\label{sec:completion}

\begin{proof}[Proof of Theorem~\ref{thm:mean-square}]
Combining Propositions~\ref{energy} and~\ref{prop:reduction},
we obtain
\begin{align*}
 \int_I|R_N(g,\alpha)-\int_\RR g(\alpha) \dd\alpha|^2\dd\alpha
 &\ll_{I,g}N^{-1/3}\log^2(2N)+N^{-2}\log^4(2N)\\
 &\ll_{I,g}N^{-1/3}\log^2(2N),
\end{align*}
which completes the proof.
\end{proof}

\begin{proof}[Proof of Theorem~\ref{thm:main}]

Put \(N_j=\lfloor j^4\rfloor\). Theorem~\ref{thm:mean-square} and Chebyshev's inequality, together with the first Borel--Cantelli lemma, yield
\[
\lim_{j\to\infty}R_{N_j}(g,\alpha)=\int_{\mathbb{R}}g(\alpha)\dd\alpha
\]
for almost all \(\alpha\in I\).

Since \(N_{j+1}/N_j\to1\), we may pass from this subsequence to the full sequence. To this end, we approximate the indicator function \(\mathbf{1}_{[-s,s]}\) by smooth compactly supported even functions, verify the interpolation step for such smooth functions, and then pass to the limit (see \cite[Section 7]{TY} for the standard interpolation argument). This proves
\[
\lim_{N\to\infty}R_{N}(s,\alpha)=2s
\]
for almost all \(\alpha\in I\).

Now apply this argument on each interval
\[
I_q:=[1/q,q],\qquad q=2,3,\dots
\]
Since \(I_q\uparrow(0,\infty)\), we obtain the desired assertion for almost every \(\alpha>0\). More precisely, for each \(q\) there is a null set \(E_q\subset I_q\) such that the convergence holds for every \(\alpha\in I_q\setminus E_q\). Hence
\[
E:=\bigcup_{q=2}^{\infty}E_q
\]
is also a null set, and the convergence holds for every \(\alpha\in(0,\infty)\setminus E\). Finally, using the symmetry
\[
R_N(s,-\alpha)=R_N(s,\alpha),
\]
we deduce the same convergence for almost every \(\alpha<0\). Therefore the assertion holds for almost every \(\alpha\in\mathbb{R}\).

\bigskip
\noindent {\bf  Acknowledgements.}  This work is supported by National Key R$\&$D Program of China (No. 2024YFA1013700) and NSFC (Nos. 12671113 and 12331005), and the Fundamental Research Funds for the Central Universities.

\end{proof}

\bigskip

\noindent School of Mathematics and Statistics, Huazhong University of Science and Technology, Wuhan 430074, China.

\noindent\textit{Email address:} \texttt{tanbo@hust.edu.cn}

\bigskip

\noindent School of Mathematics and Statistics, Huazhong University of Science and Technology, Wuhan 430074, China.

\noindent\textit{Email address:} \texttt{zhangzhen\_@hust.edu.cn}
\bigskip

\noindent School of Mathematics and Statistics, Wuhan University of Technology, Wuhan 430070, China.

\noindent\textit{Email address:} \texttt{zhouql@whut.edu.cn}

\end{document}